\documentclass[11pt,a4paper]{article}
\usepackage[dvips]{graphicx}
\usepackage[utf8]{inputenc}
\usepackage{enumerate}
\usepackage[spanish,english]{babel}
 \usepackage{amsmath}
\usepackage{amsthm,amssymb,amsfonts}
\usepackage[ruled,linesnumbered]{algorithm2e}
\usepackage{lineno}
\usepackage{hyperref}
\usepackage{caption}
\usepackage{subcaption}
\usepackage{xcolor}
\usepackage{cite}
\usepackage{pifont}
\usepackage{enumerate}
\usepackage{appendix}
\usepackage[T1]{fontenc}
\usepackage{placeins}
\usepackage{nomencl}
\makenomenclature

\graphicspath{{figures/}}

\newtheorem{theorem}{Theorem}[section]

\newtheorem{proposition}[theorem]{Proposition}

\newtheorem{lemma}[theorem]{Lemma}
\newtheorem{corollary}[theorem]{Corollary}

\newtheorem{problem}[theorem]{Problem}

\theoremstyle{remark}
\newtheorem{remark}{Remark}[section]

\title{Domination on vertex--weighted graphs induced by a coloring}
\author{María A. Garrido-Vizuete$^1$ \and Mucuy--kak Guevara$^2$\and Alberto Márquez$^1$\and  Rafael Robles$^1$}

\begin{document}


\nomenclature[alpha]{$\alpha(G)$}{Independence number of \( G \). It is the size of the largest independent set in \( G \).}
\nomenclature[Gamma_uc]{$\Gamma_{uc}(G)$}{Up--color domination number for all possible colorings of \( G \).}
\nomenclature[chi]{$\chi(G)$}{Chromatic number of \( G \). It is the minimum number of colors needed to color \( G \) such that adjacent vertices have different colors.}
\nomenclature[chi_uc]{$\chi_{uc}(G)$}{Chromatic up--color number, representing the minimum number of colors needed for an up--color dominating set.}
\nomenclature[delta]{$\delta^-(u)$, $\delta^+(u)$}{In--degree and out--degree of vertex \( u \) in a digraph derived from a colored graph.}
\nomenclature[GVE]{$G(V, E)$}{A graph \( G \) with a set of vertices \( V \) and a set of edges \( E \).}
\nomenclature[gamma]{$\gamma(G)$}{Minimum domination in a graph \( G \). It is the size of the smallest dominating set.}
\nomenclature[Kr]{$K_{r,s}$}{Complete bipartite graph with \( r \) and \( s \) vertices in each of its independent sets.}
\nomenclature[gamma_uc]{$\gamma_{uc}(G, c)$}{Up--color domination number for a pair \( (G, c) \), where \( c \) is a coloring of \( G \). It is the minimum size of a set of vertices that dominates other vertices with lower colors.}
\nomenclature[Mc]{$\mathcal{M}_c(G)$}{Set of local maxima in the graph \( G \) with respect to coloring \( c \).}
\nomenclature[iG]{$i(G)$}{Independent domination number of \( G \). It is the minimum size of a dominating set that is also independent.}
\nomenclature[N]{$N^-(u)$, $N^+(u)$}{In-neighborhood and out--neighborhood of vertex \( u \).}
\nomenclature[omega_uc]{$\omega_{uc}(G, c)$}{Up--color domination weight of a graph \( G \) given a coloring \( c \). This is the minimum weight of a set of vertices dominating other vertices with lower colors.}
\nomenclature[Omega_uc]{$\Omega_{uc}(G)$}{Minimum up--color domination weight among all possible colorings of \( G \).}
\nomenclature[PcT]{$P_n$}{Path graph with \( n \) vertices.}
\nomenclature[Cn]{$C_n$}{Cycle graph with \( n \) vertices.}
\nomenclature[ScT]{$S_c(dT)$}{Set of sink vertices in the directed tree derived from \( T \).}
\nomenclature[CcT]{$C_c(dT)$}{Set of confluence vertices in the directed tree derived from \( T \).}
\nomenclature[dT]{$dT_v$}{Directed subtree rooted at vertex \( v \).}
\nomenclature[theta]{$\theta(G)$}{Clique covering number of \( G \). It is the minimum number of cliques needed to cover all vertices of \( G \).}


\maketitle

\footnotetext[1]{Dep. Matem\'atica Aplicada I, Universidad de Sevilla, Spain. Emails: \{vizuete, almar, rafarob\}@us.es}
\footnotetext[2]{Dep. Matem\'aticas, Facultad de Ciencias, UNAM, México. Email: mucuy-kak.guevara@ciencias.unam.mx}

\begin{abstract}
    This paper introduces the concept of domination in the context of colored graphs (where each color assigns a weight to the vertices of its class), termed {\em up--color domination}, where a vertex dominating another must be heavier than the other. That idea defines, on one hand, a new parameter measuring the size of minimal dominating sets satisfying specific constraints related to vertex colors. The paper proves that the optimization problem associated with that concept is an NP--complete problem, even for bipartite graphs with three colors. On the other hand, a weight-based variant, the up--color domination weight, is proposed, further establishing its computational hardness. The work also explores the relationship between up--color domination and classical domination  and coloring concepts. Efficient algorithms for trees are developed that use their acyclic structure to achieve polynomial--time solutions.

\end{abstract}

\section{Introduction} \label{sec:intr}
As is well known, the domination number $\gamma(G)$ of a graph $G$ is the size of a smallest set $D$ of vertices of $G$ such that every vertex outside $D$ has at least one neighbor in $D$, that is, $D$ is a dominating set of $G$.
Furthermore, there are several papers dealing with the problem of domination on colored graphs (see \cite{2012-abc-odcig,2019-ms-pdcng,2020-alikhani-total, 2023-hkpp-ctdcg}), in which a domination coloring of a graph $G$ is a coloring of $G$ such that each vertex of $G$ dominates at least one color class, and each color class is dominated by at least one vertex. However, this definition is too restrictive and it does not reflect among the vertices of the graph the usual hierarchy given by a domination relationship. 
Furthermore, there is the concept of Roman domination~\cite{Stewart1999DefendTR,COCKAYNE200411,ReVelle2000}, where the vertices are colored with the colors $0,1,2$, such that the vertices with color $2$ dominate vertices of color $0$, respecting the hierarchy of color $2$ over $0$, but not including color $1$ in this hierarchy.

In this work, we introduce a natural concept of domination in the context of colored graphs. Throughout this paper, a graph is finite, undirected, and simple and is denoted by $G(V,E)$, where $V$ and $E$ are its set of vertices and its set of edges, respectively.

Given a graph $G(V,E)$, a coloring of $G$ is a map $c:V \longrightarrow \{0,1,2\ldots\}$ such that $c(v_i) \neq c(v_j)$ if $v_i$ and $v_j$ are adjacent vertices. If the coloring uses exactly $k$ colors, it is called a $k$--coloring, and, as is well known, the chromatic number $\chi(G)$ of a graph $G$ is the minimum value of $k$. Hereinafter, an optimal coloring is a $\chi(G)$--coloring.

Then, given a graph $G$ with a coloring $c$, it is said that $D \subseteq V$ is an {\em up--color dominating $c$--set} of the pair $(G,c)$ if: (1) for any vertex $v$ not in $D$ there exists an adjacent vertex $d \in D$ such that $c(v)<c(d)$; (2) $D$ contains no vertex of color $0$. 

Although our study focuses on graphs, digraphs are handled as a tool for representing the up--color domination relationship. Any coloring $c$ of $G(V,E)$ provides a natural (partial) ordering of $V$, so the pair $(G,c)$ can be considered as an acyclic digraph such that every edge starts from its vertex with the highest color. Thus, all studies and results about out--domination in acyclic digraphs are valid for up--color domination in graphs.

As usual in this directed context, $\delta^-(u)$ and $N^-(u)$ denote the in--degree and in--neigh\-bor\-hood of the vertex $u$, respectively, and, analogously, $\delta^+(u)$ and $N^+(u)$ are the out--degree and out--neighborhood of $u$, respectively. Recall that a {\em confluence vertex} $v$ verifies that $\delta^-(v)>1$. In this paper, even if we do not mention explicitly, we will consider that any graph with a coloring as mentioned before has the structure of a digraph.

The structure of this paper is as follows. After this introductory section, we give some definitions and some easy bounds in the next section. As we said, we try to optimize the size or the weight of an up--color dominating set; the first task is accomplished in Section 3, where we prove that the problem is NP--complete even for bipartite graph (and three colorings), but not for trees. And, after proving that for any graph and for any dominating set, there exists always a coloring such that that set is also an up--color domination set, we define a new problem (to find an optimal coloring with such a property) that turns out to be NP--hard. Then, in the same section, we study some properties of the size of optimal up--color dominating sets. Similar problems, but with respect to the weight of the set instead of its size, are considered in Section 4. Again, we give some results about complexity, comparing the results on general graphs with those of trees. We finish with a section with some conclusions and open problems. For the sake of clarity, we have included an appendix with the notation used along the paper.

\section{Definitions and preliminaries}

This section is devoted to the basic concepts and problems arising in up--color domination context. Likewise, some preliminary results are presented, the proofs of which can be obtained straightforwardly.

Graph-theoretic concepts and definitions used in this work are rooted in the foundational text by Harary~\cite{1969-h-gt}, which provides an extensive exploration of graph theory basics.

As we defined above, given a graph $G$ with a coloring $c$, it is said that $D \subseteq V$ is an {\em up--color dominating $c$--set} of the pair $(G,c)$ if: (1) for any vertex $v$ not in $D$ there exists an adjacent vertex $d \in D$ such that $c(v)<c(d)$; 
(2) $D$ contains no vertex of color $0$. Thus, a up--color domination relationship naturally arises among the vertices of $G$, through which $v$ and $d$ are said to be {\em submissive} and {\em dominant} vertices, respectively.

This domination concept allows us to define two different optimization goals. In this way, following tradition, we can minimize the size of the dominating set and thus the smallest cardinal of an up--color dominating $c$--set for the pair $(G,c)$ is called {\em the up--color domination $c$--number} of $(G,c)$ and is denoted by $\gamma_{uc}(G,c)$. Furthermore, insofar as the coloring gives in a natural way a weight to the vertices of the graph (and to subsets of the vertices, adding the individual weights), the minimum weight of an up--color dominating $c$--set for the pair $(G,c)$ is called {\em the up--color domination $c$--weight} of $(G,c)$ and is denoted by $\omega_{uc}(G,c)$. 

The vertex sets that lead to $\gamma_{uc}(G,c)$ and $\omega_{uc}(G,c)$ are called {\em  $\gamma_{uc}$--dominating $c$--set} and {\em $\omega_{uc}$--dominating $c$--set}, respectively. Obviously, both any $\gamma_{uc}$-- and $\omega_{uc}$--dominating $c$--set contains all vertices colored with the highest color of $c$.

Furthermore, the minimum $\omega_{uc}(G,c)$ among all possible colorings $c$ is called {\em the up--color domination weight of $G$} and denoted by $\Omega_{uc}(G)$. Every coloring $c$ verifying $\Omega_{uc}(G)=\omega_{uc}(G,c)$ is named {\em $\Omega_{uc}$--domination coloring} and the associated $\omega_{uc}$--dominating $c$--set is called  {\em $\Omega_{uc}$--dominating set}. In Section~\ref{sec:cwd}, we address the question of weight in the context of up--color domination in graphs. While in Section ~\ref{sec:cdn}, we focus on providing bounds for the up-color domination number and calculating the complexity of its computation.

The definition of up--color domination number makes sense: consider $K_{2,3}$, if we assign to the class with three elements the largest number in a $2$--coloring $c$, then $\gamma_{uc}(K_{2,3},c)=3$, but if we interchange the colors, then $\gamma_{uc}(K_{2,3},c')=2=\gamma(G)$.

A very first and simple result is:

\begin{lemma} \label{lem:1bound}
Given a pair $(G,c)$, $\gamma_{uc}(G,c) \geq \max \{ |c^{-1}(i)|, \gamma(G) \} $, where $i$ represents the largest color with non--empty inverse image.
\end{lemma}

In fact, it is possible to refine a little bit previous lemma. Given a pair $(G,c)$, a vertex $v\in V(G)$ is a {\em local maximum} for $c$ if 
$c(u) \leq c(v)$ for all neighbor $u$ of $v$. Let $M_c(G)$ denote the set of local maxima of $G$ for $c$ (obviously, $c^{-1}(i) \subseteq M_c(G)$). Now:

\begin{lemma} \label{lem:1bound2}
Given a pair $(G,c)$, $\gamma_{uc}(G,c) \geq \max \{ |M_c(G)|, \gamma(G) \}$.
\end{lemma}

As an immediate consequence of Lemma \ref{lem:1bound} (or Lemma \ref{lem:1bound2}) we can find examples of pairs $(G,c)$ such that $\gamma_{uc}(G,c)>\gamma(G)$. For instance, consider $K_{r,s}$, $r\geq s>2$, with $c$ any $2$--coloring. Then, clearly, $\gamma_{uc}(G,c)\geq s>2=\gamma(G)$.

\section{Up--color domination number} 
\label{sec:cdn}

Dominator coloring and total dominator coloring are concepts that combine domination theory and graph coloring. These notions have been extensively studied by Henning~\cite{2021-henning-dtdcig}, providing a foundational theoretical framework for the analysis presented in this work.

In this section, we consider the problem of optimizing the size of an up--color dominating set, and we prove that this problem is NP--complete even for bipartite graph with three colors, but linear for trees. Then, after proving that for any graph and for any dominating set, there exists always a coloring such that the same set is also an up--color dominating set, we define a new problem (to find an optimal coloring with such a property) that turns out to be NP--hard. Finally, we provide some theoretical bounds for the parameters considered in this section.

\noindent \textbf{\textsc{up--color domination $c$--number}}  \\
{\sc instance:} A graph $G=(V,E)$, a coloring $c$ and a positive integer $k$. \\
{\sc question:} Is $\gamma_{uc}(G,c) \leq k$?

It is trivial to prove that if $c$ uses only two colors ($1$ and $2$), then the $\omega_{uc}$--dominating $c$--set is just $c^{-1}(2)$. Thus, using only two colors $\gamma_{uc}(G,c)$ can be computed in linear time. That is not the case if $c$ uses at least three colors, even if $G$ is a bipartite graph. 

\begin{theorem} \label{th:npcd}
\textsc{\emph{up--color domination $c$--number}} is an NP--complete problem, even if the graph is bipartite and the image of the coloring is $\{1,2,3 \}$, and even for optimal coloring. 
\end{theorem}
\begin{proof}

Obviously, \textsc{up--color domination number} is in NP.

Next, we transform \textsc{minimum cover} to \textsc{up--color domination number} 
Consider an instance of \textsc{minimum cover}. Thus, let $C$ be a collection of subsets (with at most three elements each) of a finite set $S$. If $|C|=l$, we map each element of $S$ to $l$ vertices of a graph $G$. Additionally, we introduce an extra vertex for every subset in $C$ and an additional vertex, $v_0$. This results in a graph $G$ (as depicted in Fig. \ref{fig:thnpcd}) where edges connect these vertices to the corresponding elements in the set and further edges link the extra vertex with all vertices representing subsets in collection $C$.

Now  consider the following coloring:
$$
c(v)= \left\{ \begin{array}{lc}
             1 &   \textrm{if $v$ represents an element of } S\\
             2 & \textrm{if $v$ represent an element of $C$}\\
             3 & \textrm{if } v=v_0
             \end{array}
   \right.$$
   
Clearly, $v_0 $ must belong to any up--color dominating $c$--set and the correspondence between choosing some subsets in $C$ and adding the related vertices of $G$ is straightforward.

Note that the coloring of the graph resulting from this particular transformation reveals is not an optimal coloring, however, we can add a dummy triangle to obtain an optimal coloring (since the new graph needs $3$ colors).
\end{proof}

\begin{figure}[ht]
	\centering
\includegraphics[width=0.8\textwidth]{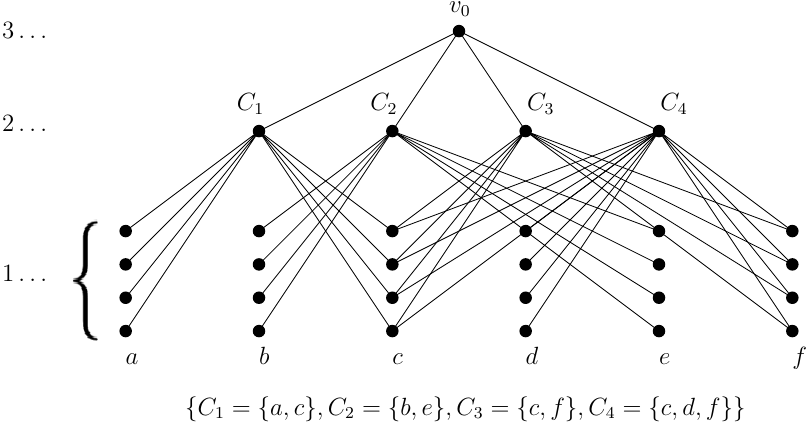}
	\caption{Theorem \ref{th:npcd}.}
 \label{fig:thnpcd}
\end{figure}

On the other hand, we have:

\begin{proposition} \label{prop:equal}
For any graph $G$ there exists a coloring $c$ such that $\gamma_{uc}(G,c)=\gamma(G)$.
\end{proposition}
\begin{proof} Let $c_0$ be a coloring of a graph $G=(V,E)$. If $\max_{\{v\in G\}}c_0(v)=k$, and considering a dominating set $D$ of $G$, we can define the following coloring of $G$. First of all, we rearrange $c_0$ in such a way that $c_0^{-1}(k)=\{v \in D: c_0(v)=k\} \neq \emptyset $

$
c(v)= \left\{ \begin{array}{lcc}
             c_0(v)+k &   \textrm{if}  & v \in D \setminus c^{-1}(k)\\
             c_0(v) &  & \textrm{otherwise}
             \end{array}
   \right.$.

And it is straightforward to check that $c$ is a coloring and $D$ is an up--color dominating $c$--set.
\end{proof}

The smallest number of colors, $\chi_{uc}(G)$, needed to preserve $\gamma_{uc}(G,c)=\gamma(G)$ will be called {\em the chromatic up--color domination number} of $G$. As a direct consequence of the proof of Proposition \ref{prop:equal}, we obtain:

\begin{corollary} \label{cor:chid}
For any graph $G$, $\chi(G) \leq \chi_{uc}(G) \leq 2 \chi (G)-1 $.
\end{corollary}

\begin{remark}
The bounds given in Corollary \ref{cor:chid} are sharp. Establishing examples illustrating the trivial nature of the lower bound is straightforward. To demonstrate the tightness of the upper bound, consider the graph $G$ constructed as follows: to each vertex of $K_{n+1}$ (excluding one vertex denoted as $v$), attach a $K_n$. The only optimal dominating set $D$ is the set defined by all the vertices but $v$ of the original $K_{n+1}$. Now it is straightforward to check that $\chi_{uc}(G) = 2 \chi (G)-1 $.
\end{remark}

The study of graph colorability and its computational complexity has been a cornerstone in graph theory. Determining whether a graph is \( k \)-colorable or finding optimal colorings are well-known NP-complete problems. Such problems serve as a foundation for exploring the complexity of related graph parameters, including domination in colored graphs. Building on this perspective, the following problem delves into the NP-hardness of the up–color domination number.

\noindent \textbf{\textsc{chromatic up--color domination number}} \\
{\sc instance:} A graph $G=(V,E)$ and a positive integer $k$. \\
{\sc question:} Is $\chi_{uc}(G) \leq k$?

And the following theorem focuses on the analysis of the complexity associated with solving the previous problem.

\begin{theorem} \label{th:npchi}
\textsc{\emph{chromatic up--color domination number}} is an NP--hard problem. Even for $k=3$.
\end{theorem}

\begin{proof} 

We are going to make a reduction from \textsc{3--sat}. Given an instance $\phi$ of \textsc{3--sat}, let $c_1, c_2, \ldots, c_s$ be the clauses defined over variables $\{u_1,u_2, \ldots, u_r \}$. We are going to build the same graph used to prove the NP--hardness of 3--coloring.
The gadgets employed are analogous to those presented by Garey, Johnson, and Stockmeyer~\cite{Garey1976} in their pioneering work on NP-complete problems.  
The components of that graph $G=(V,E)$ are the following: first of all, we first create a triangle with its vertices labeled as $\{T, F, B\}$ where $T$ stands for True, $F$ for False and $B$ for Base. Then, we add two new vertices $u_i$ and $\overline{u_i}$ for every variable $u_i$ and join with edges these two vertices and with the original vertex labeled as $B$ (obtaining, in this way $r$ new triangles, one for each variable), see Figure \ref{fig:chi-var}. 
Now, we are going to create a new gadget for each clause. That gadgets are constructed as Figure \ref{fig:chi-clau} shows (and Figure \ref{fig:chi-int} represents how the whole graph $G$ looks like). 

Now, it is not hard to check that $\chi(G)=\min \{2s+1,2r+s+1 \}$. In the first case, we need at least one vertex for each one of the $2s$ triangles appearing in the clause--gadgets plus an additional vertex --the one labeled as $B$-- to dominate the remaining vertices. In the second case, we can choose for the minimum dominating set all the $2r$ vertices obtained from variables, plus one vertex for each clause--gadget, and the vertex $F$ of the first triangle $\{T,F,B\}$.  

First of all, assume $\min \{ 2s+1,2r+s+1\}=2s+1$. That means, as we said before, that a minimum dominating set is just choosing one vertex for each one of the $2s$ triangles appearing in the clause--gadget plus the vertex labeled as $B$. But it is not difficult to see that if $G$ is $3$--colorable (let's say with colors $\{0,1,2\}$, then we can assign the color $2$ just to those vertices (there is not a triangle with two of its vertices in the dominating set). On the other hand, the same happens when $\min \{ 2s+1,2r+s+1\}=2r+s+1$, and we can assign, if $G$ is $3$--colorable, the color $2$ to the vertices in the minimum dominating set. 

Thus, in both cases, if the graph $G$ is $3$--colorable, we can assign the color $2$ to just the vertices in the minimum dominating set, but to decide if that kind of graph is $3$--colorable is equivalent to solve the initial instance of \textsc{3--sat}.

      \begin{figure}[ht] 
	\centering
\includegraphics[width=0.6\textwidth]{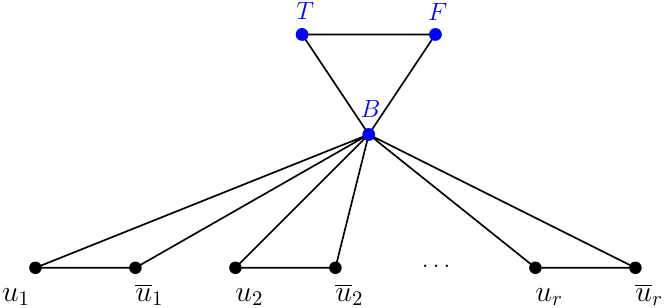}
	\caption{The gadget for the variables. Theorem \ref{th:npchi}.}
 \label{fig:chi-var}
\end{figure}

\begin{figure}[ht] 
	\centering
\includegraphics[width=0.55\textwidth]{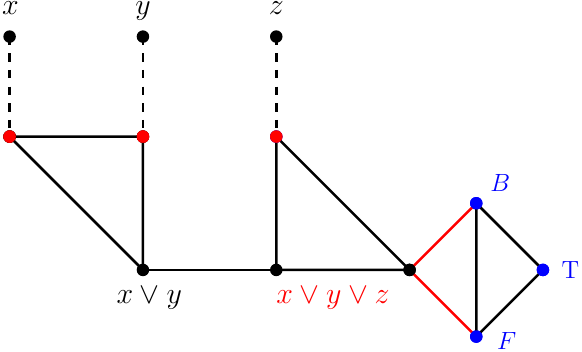}
	\caption{The gadget for the clauses. Theorem \ref{th:npchi}.}
 \label{fig:chi-clau}
\end{figure}

\begin{figure}[ht] 
	\centering
\includegraphics[width=1\textwidth]{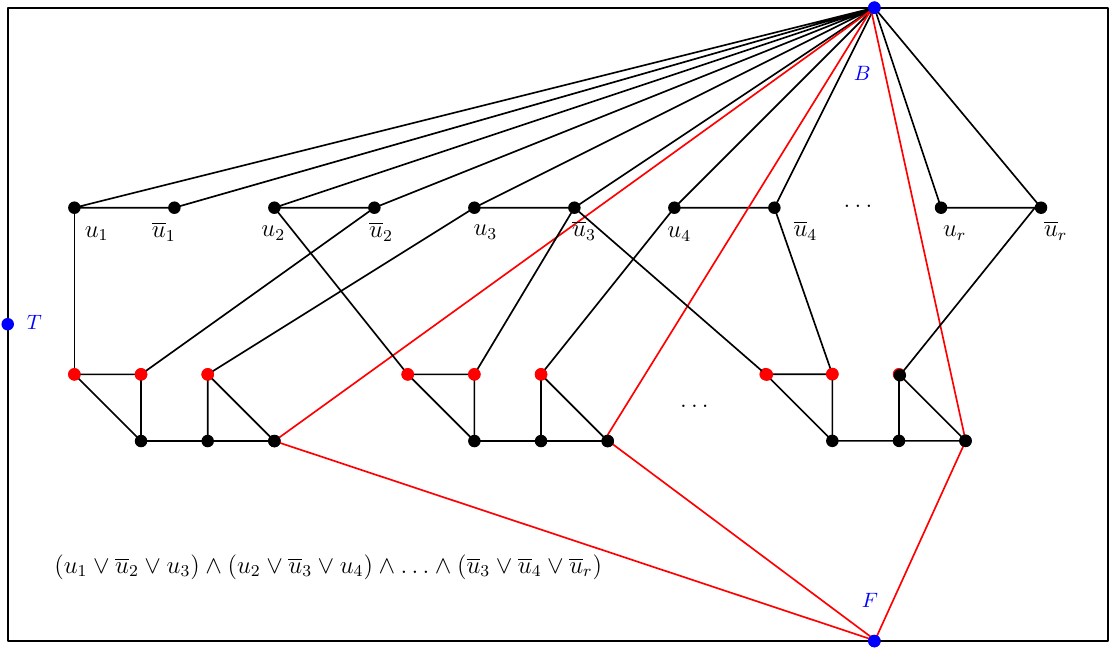}
	\caption{An example of integrated gadgets. Theorem \ref{th:npchi}.}
 \label{fig:chi-int}
\end{figure}
    
\end{proof}

Now some bounds for $\gamma_{uc} (G,c)$ are provided.

\begin{proposition} \label{prop:ineq}
For any graph $G$ and any coloring $c$, $\gamma(G) \leq \gamma_{uc} (G,c) \leq \chi (\bar{G})$. Where $\bar{G}$ represents the complement of $G$.
\end{proposition}
\begin{proof}
The first inequality is trivial by definition, and the second inequality follows from the fact that  for any clique in $G$ and any coloring $c$, we only need to choose the vertex with the highest color in that clique (recall that the clique covering number $\theta(G)$ of a graph $G$ is the minimum number of cliques in $G$ needed to cover the vertex set of $G$. And that number is also given by $\theta(G)=\chi(\bar{G})$).
\end{proof}

The independent domination number of $G$, denoted by $i(G)$, is the minimum size of an independent dominating set.  It follows immediately that $\gamma (G) \leq i(G) \leq \alpha (G)$, where $\alpha (G)$ is the independece number. See \cite{2013-gh-idgsrs} for a survey on independent domination in graphs.

We know compute those parameters considered so far for some common graph families.

\begin{proposition} \label{prop:path}
Let $P_n$ be the path with $n$ vertices. Then, 
\begin{enumerate}
    \item For any coloring $c$, $\lceil n/3 \rceil \leq \gamma_{uc}(P_n,c)\leq \lceil n/2\rceil$.
    \item For any optimal coloring $c$, $\lfloor n/2 \rfloor \leq \gamma_{uc}(P_n,c)\leq\lceil n/2\rceil$.
    \item $\chi_{uc}(P_n)=3$ ($n\geq 5$).
    \end{enumerate}
\end{proposition}

\begin{proof}
1. 
Remind $\gamma (P_n) = i(P_n)= \lceil n/3 \rceil$ (and $\alpha (P_n)=\lceil n/2 \rceil$), thus, by using Lemma \ref{lem:1bound}, we have the lower bound. For the upper bound, we can match each pair of consecutive vertices and choose, in each pair, that one with the highest color.

2. If $c$ is an optimal coloring, both chromatic classes have, at least, $\lfloor n/2 \rfloor$ vertices and the class with the highest color is, in this case, a $c$--dominating set.

3. Obviously, $\chi_{uc}(P_n)>2$. However, for any optimal coloring of $G$, we can match the vertices in groups of $3$ and assign an extra color to the central vertex of each group.
\end{proof}

\begin{proposition} \label{prop:cycle}
Let $C_n$ be the cycle with $n$ vertices. Then, for any coloring $c$, $\lceil n/3 \rceil \leq \gamma_{uc}(C_n,c)\leq \lfloor n/2\rfloor$, and $\chi_{uc}(C_n)=3$.

\end{proposition}

\begin{proof}
The proof is similar to Proposition \ref{prop:path} keeping in mind 
$\gamma (C_n) = i(G)= \lceil n/3 \rceil$ and $\alpha (G)=\lfloor n/2 \rfloor$.
\end{proof}

\begin{proposition} \label{prop:bipartite}
Let $K_{r,s}$ be the complete bipartite graph with $r+s$ vertices. Then, 
\begin{enumerate}
    \item For any coloring $c$, $2 \leq \gamma_{uc}(K_{r,s},c)\leq \max \{ r,s \}$.
    \item For any optimal coloring $c$, $  \gamma_{uc}(K_{r,s},c) \in \{\min \{ r,s \},\max \{ r,s \}\}$.
    \item $\chi_{uc}(K_{r,s})=3$
\end{enumerate}
\end{proposition}
\begin{proof}
The proof is straightforward by using Lemma \ref{lem:1bound} and the facts  $\gamma (K_{r,s}) = 2$, $ i(K_{r,s})= \min \{ r,s \}$ and $\alpha (G)=\max \{ r,s \}$.
\end{proof}

Observe that, in the three cases, for any optimal coloring $\gamma_{uc}(G)$ takes, at most, only two values. Of course, this is not valid for any graph, for example, the graph depicted in Figure \ref{fig:casita} has optimal colorings where the classes with the highest color can have $2$, $3$ or $4$ elements and in the three cases that class is also dominant.

\begin{figure}[ht] 
	\centering
\includegraphics[width=0.7\textwidth]{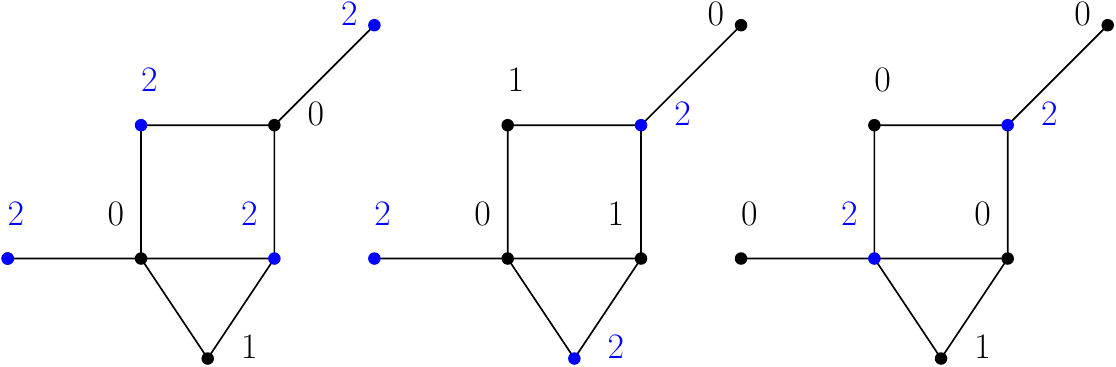}
	\caption{Three optimal colorings for the same graph.}
 \label{fig:casita}
\end{figure}


Next, we study the graph class of trees and prove that it is possible to obtain $\gamma_{uc}$--dominating $c$--sets in linear time. 

First of all, It is important to note that, even in a tree, the number of $\gamma_{uc}$--dominating $c$--set can be exponentially large as Figure \ref{fig:domsetexptree} shows. In spite of that fact, it is possible to find a $\gamma_{uc}$--dominating $c$--set in linear time as we will see next. Firstly, we need some technical lemmas.

\begin{figure}[ht] 
	\centering
\includegraphics[width=0.7\textwidth]{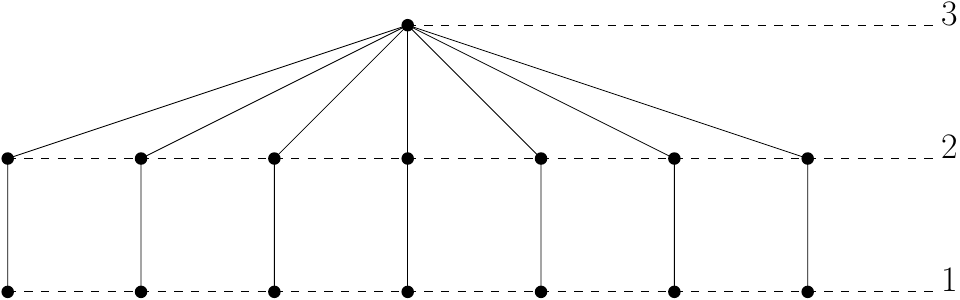}
	\caption{In this tree, the root, with color $3$, must be in any $\gamma_{uc}$--dominating $c$--set, but then, it is possible to include either each vertex with the color $2$ or its child. So, the number of $\gamma_{uc}$--dominating $c$--sets is exponential.}
 \label{fig:domsetexptree}
\end{figure}

\begin{remark} \label{remark:localmax}
    Obviously, if $D$ is a $\gamma_{uc}$--dominating $c$--set and $M$ is the set of local maxima in $(T,c)$, then $M \subseteq D$.
\end{remark}

\begin{lemma} \label{lem:treedom1}
Given a pair $(T,c)$ where $T$ is a tree, there exists always a $\gamma_{uc}$--dominating $c$--set with no leaves other than the possible local maxima.
\end{lemma}

\begin{proof}
    Let $D$ a $\gamma_{uc}$--dominating $c$--set with leaves which are not local maximum and the minimum number of these leaves possible.
    Let $v\in D$ leave and not local maximum, then $N(v)=\{u\}$, such that $c(u)>c(v)$. Now the set $D'=(D\setminus\{v\})\cup \{u\}$ is also a $\gamma_{uc}$--dominating $c$--set, but with less number of leaves not local maximum, contradicting the election of $D$.

\end{proof}

\begin{lemma}\label{lem:treedom2}
    Given a pair $(T,c)$ where $T$ is a tree, there exists always a $\gamma_{uc}$--dominating $c$--set $D$ such that if $v \in D$ and $N^-(v) \cap D \neq \emptyset$, then $\delta ^+(v)>1$.  
\end{lemma}

\begin{proof}
    Let $D$ a $\gamma_{uc}$--dominating $c$--set such that there exists vertex $v\in D$ with $N^-(v)\cap D\neq \emptyset$ and $\delta^+(v)=0$. We suppose that $D$ has the minimum number of this kind of vertex among all the $\gamma_{uc}$--dominating $c$--sets. 
    Let $v\in D$ with $N^-(v)\cap D\neq \emptyset$ and $\delta^+(v)=0$ and $u\in N^-(v)\cap D$. Now the set $D'=(D\setminus\{v\})\cup \{u\}$ is also a $\gamma_{uc}$--dominating $c$--set, but with less number of vertices holding the required condition, contradicting the election of $D$.

\end{proof} 

\begin{theorem} \label{th:tree_lineal1}
	
	Given a pair $(T,c) $ where $T$ is a tree and $c$ is a coloring, it is possible to find a {\em  $\gamma_{uc}$--dominating $c$--set} $D$ in linear time.
	
\end{theorem}

\begin{proof}
To achieve this, an algorithm is described that finds the up–color dominating set in linear time. We define 
$D$ as the set of vertices to which we will assign the role of dominant, and $S$ as the set of vertices that will take on the non--dominant or submissive role. The algorithm will add vertices to these sets until all vertices in $T$ have been assigned a role. The steps of the algorithm are as follows:

In a first step we compute the in--degrees and the out--degree of each vertex (remember that we can consider the edges of the graph as oriented from the vertex of higher color to the vertex of lower color). Now, 
\begin{itemize}
    \item[(a)] Following Remark \ref{remark:localmax} we label all local maxima as in $D$.
    \item[(b)] Following Lemma \ref{lem:treedom1}, we assign all the leaves that are not to $D$ as in $S$, and their parents as in $D$.
    \item[(c)] Following Lemma \ref{lem:treedom2}, we include in $S$ to any vertex $v$ such that $N^{-}(v)\cap D\neq \emptyset$ \textbf{and} $\delta^{+}(v)\leq 1$.
\end{itemize}

At this point, we must pointed out that every time we add a vertex $v$ to $D$, we prune from $T$ all its incoming edges, and every time we consider a vertex $u$ as an element of $S$, we prune from $T$ all its incident edges. Of course, updating the degrees after those deletion costs a linear compensated time (the number of those updates depends only on the sum of the degrees of the vertices, that is $2n-2$ if $n$ is the number of vertices in $T$). Thus, in order to guarantee a total linear time, we next check that each time we can find at least a vertex to assign a role. 

Let's see that if there are vertices with no assignation after some steps, then we can find always one vertex satisfying (b) or (c). 
According to the description of the algorithm, the assignment of a vertex $v$ as dominant entails the deletion of its incoming edges. Consequently, the updated $T$  contains a subtree rooted at $v$. This subtree must have at least one leaf $u$ other than $v$. If $u$ is a sink then Case (b) includes $u$ in $S$. If $u$ is a source then Case (a) assigns him to $D$ but, moreover, if $u$  has an out--neighbor $w$ with out--degree  $1$ so it is assigned to $S$ in (c) and deleted from $T$. Therefore, the only situation that does not involve vertex and/or edges deletion in $T$ is the presence of a dominant vertex that each out--neighbor has at least out--degree $2$. If we mark one of that out--neigbours as the right vertex and consider the subtree rooted in it, we can repeat the reasoning on the new subtree which will contain a new subtree of the subtree and so on. Since we cannot repeat this situation indefinitely because the tree is finite, then at some point we must find a new dominant vertex with an out--neighbor that will be assigned submissive by (c) and will be deleted in $T$.

\end{proof}

Observe that algorithm described in Theorem \ref{th:tree_lineal1} does not work if we apply it to graph in general.

\begin{figure}[ht] 
	\centering
\includegraphics[width=0.3\textwidth]{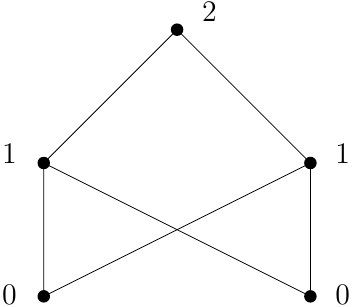}
	\caption{In this graph, algorithm described in Theorem \ref{th:tree_lineal1} labels as {\em dominant} the vertex with color $2$, but it does not label any other vertex.}
 \label{fig:alggraf}
\end{figure}

\section{Up--color domination weight} 
\label{sec:cwd}

In this section, the problem of minimizing the sum of the vertex weights in the context of up--color domination in graphs is tackled. We recall that $\omega_{uc}(G,c)$ denotes the up--color domination $c$--weight, that is,  the minimum weight of an up--color dominating $c$--set  for the pair $(G,c)$. Thus, the up--color domination weight of $G$, represented by $\Omega_{uc}(G)$, is the minimum $\omega_{uc}(G,c)$ among all possible colorings $c$, and those colorings that reach this minimum are called $\Omega_{uc}$--domination colorings.

This section begins by delving into the foundational properties of the up-color domination weight, with bipartite graphs serving as illustrative examples to clarify these concepts. Theorem~\ref{th:Omegachigamma} establishes upper and lower bounds for $\Omega_{uc}$ in connected graphs, leveraging well-known parameters such as the chromatic number, domination number, and independent domination number. Proposition~\ref{prop:hairy} focuses on the specific case of hairy graphs, analyzing the number of colors required for $\Omega_{uc}$--domination colorings and its implications for structural properties of these graphs. The section proceeds to prove the NP--completeness of two central problems: the \textsc{up--color domination $c$--weight}, via reductions from the \textsc{minimum cover} problem, and the \textsc{up--color domination weight}, through reductions from \textsc{balanced maximum E2–-satisfiability}. Additionally, a recursive algorithm is introduced to solve the \textsc{up--color domination $c$--weight} problem for trees in polynomial time, using their acyclic structure to achieve efficiency. The section concludes with bounds for $\Omega_{uc}$ specific to bipartite graphs, exploring the interplay between graph structure, colorings, and domination weight.

A first natural question is whether all $\Omega_{uc}$--domination colorings use the same number of colors. The bipartite graph $K_{3,3}$ provides a simple example of the negative answer to this question: it is straightforward to see that $\Omega_{uc}(K_{3,3})=3$, however, two $\Omega_{uc}$--domination colorings of $K_{3,3}$ that use two and three colors, respectively, are depicted in Figure~\ref{fig:k33_1}.

\begin{figure}[ht] 
	\centering
\includegraphics[width=0.5\textwidth]{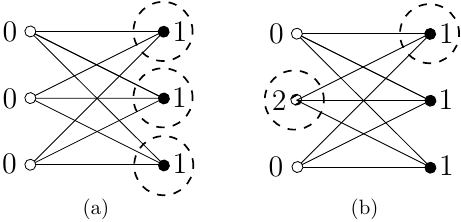}
	\caption{Two $\Omega_{uc}$--domination colorings for $K_{3,3}$ using different number of colors. The $\Omega_{uc}$--dominating sets are marked by circles with a dashed line.}
 \label{fig:k33_1}
\end{figure}

This example of complete bipartite graphs leads us to relate the up--color domination weight problem with the Grundy number of an undirected graph, that is, the maximum number of colors that can be used by a greedy coloring strategy that considers the vertices of the graph in sequence and assigns each vertex its first available color, using a vertex ordering chosen to use as many colors as possible.

Thus, the Grundy number of the complete bipartite graphs is two and all the colorings obtained by the greedy coloring algorithm are like the one shown in Figure~\ref{fig:k33_1}(a). Nevertheless, $\Omega_{uc}(K_{r,s})=3$ for $r,s \geq 4$ (see Figure~\ref{fig:k33_1}(b)), then no $\Omega_{uc}$--domination coloring can be obtained using the greedy coloring algorithm.

Regarding the relationship with Roman domination; a Roman dominating function on a graph $G=(V,E)$ is a function $f:V \rightarrow \{0,1,2\}$ satisfying the condition that every vertex $u$ for which $f(u) = 0$ is adjacent to at least one vertex $v$ for which $f(v) = 2$. The minimum weight of a Roman dominating function on a graph $G$ is called the Roman domination number of $G$, $\gamma_r(G)$ \cite{COCKAYNE200411}. Observe that $\Omega_{uc}(K_n)=n-1$ but $\gamma_r(K_n)=2$. On the other hand, $\Omega_{uc}(K_{1,n})=1$ but $\gamma_r(K_{1,n})=2$.

The following result provides bounds for the up--color domination weight based on known graph parameters.

\begin{theorem} \label{th:Omegachigamma}
Let $G$ be a connected graph of $n \geq 5$ vertices. Then, $$\max \{\chi(G)-1, \gamma (G) \} \leq \Omega_{uc}(G)\leq \min \{\frac{3}{2}(\chi (G)-1)\gamma(G), i(G)\chi(G)\} \leq \frac{1}{4}n^2.$$ Moreover, $\chi(G)-1= \Omega_{uc}(G)$ if and only if $G$ is a cone graph, that is, the closed neighborhood of one vertex is the whole graph.
\end{theorem}

\begin{proof}
By definition, the first inequality $\max \{\chi(G)-1, \gamma (G) \} \leq \Omega_{uc}(G)$ is obvious. Furthermore, if $\chi(G)-1= \Omega_{uc}(G)$, as all vertices colored with the greatest color ($\chi(G)-1$) must belong to the $\Omega_{uc}$--dominating set, then there is only one vertex of color $\chi(G)-1$, that dominates all the remaining vertices. And, given the cone graph, it is suffices to consider a $\chi(G)$--coloring that assigns the color $\chi(G)-1$ to the vertex adjacent to all the others.

About the second inequality, on the one hand, $\Omega_{uc}(G)\leq \frac{3}{2}(\chi (G)-1)\gamma(G)$ is obtained by constructing an up--color dominating $c$--set in the following way:

Let $D$ be a dominating set of $G$ of cardinal $\gamma(G)$ and $c$ an optimal coloring with chromatic number \( \chi(G) = s \). That coloring partitions the vertex set \( V \) into \( s \) color classes \( C_0, C_1, \dots, C_{s-1} \), such that \( V = \bigcup_{i=0}^{s-1} C_i \) and \( C_i \cap C_j = \emptyset \) for \( i \neq j \). Denote the intersection of \( D \) with each color class as \( D \cap C_i \). There exists an ordering \( (i_0, i_1, \dots, i_{n-1}) \) of the color classes such that:
\[
|D \cap C_{i_j}| \geq |D \cap C_{i_k}| \quad \text{for } j < k.
\]

The colors of the classes and the colors of the vertices in the dominating set $D$ are reassigned as follows:

1. The class \( C_{i_0} \) is assigned color \( s-1 \).

2. The remaining classes \( C_{i_1}, C_{i_2}, \dots, C_{i_{s-1}} \) are assigned colors \( 0, 1, \dots, (s-2) \) in any order.

3. For the elements of \( D \), the following additional assignments are made:

   - Elements of \( D \cap C_{i_1} \) are assigned color \( s \).
   
   - Elements of \( D \cap C_{i_2} \) are assigned color \( s+1 \).
   
   - In general, elements of \( D \cap C_{i_k} \) are assigned color \( s+k-1 \) for \( k = 1, \dots, s-1 \).

Let $c'$ be the resulting new coloring of $G$.
To organize the new colors of the elements in 
$D$, we define a matrix $M$ as follows:
Define a matrix \( M \) of size \( s \times m \), where \( m = |D \cap C_{i_0}| \). Each entry \( m_{k,j} \) of the matrix is defined as:
\[
m_{k,j} =
\begin{cases}
\text{the color of the } j\text{-th element of } D \cap C_{i_k}, & \text{if it exists;} \\
0, & \text{otherwise.}
\end{cases}
\]

The sum of the elements in the \( j \)-th column of \( M \), (it is the sum of the elements of an arithmetic progression), is denoted as:
\[
w_j = \sum_{k=0}^{s-1} m_{k,j}= (s-1)+s+(s+1)+\dots.
\]
If each column contains \( n_j \) non-zero elements, we have:
\[
w_j \leq \frac{(s-1) + (2s-2)}{2} \cdot n_j.
\]

Summing over all columns \( j \), we obtain the weight of \( D \) under the new coloring:
\[
\omega_{uc}(G,c') \leq \sum_{j=1}^m w_j.
\]
Since \( \sum_{j=1}^m n_j = |D|=\gamma(G) \), the total weight is bounded by:
\[
\sum_{j=1}^m w_j \leq \frac{(s-1) + (2s-2)}{2} \cdot \gamma(G).
\]
Simplifying further:
\[
\Omega_{uc}(G)\leq\omega_{uc}(G,c')\leq\sum_{j=1}^m w_j \leq \frac{3}{2}(s-1) \cdot \gamma(G).
\]

On the other hand, the inequality $\Omega_{uc}(G)\leq i(G)\chi(G)\}$ is obtained in a similar way, but being $D$ a minimal independent dominating set of $G$, and coloring its vertices with the color $\chi(G)$. 

For the last part we prove that $i(G)\chi(G) \leq \frac{1}{4}n^2$ using a similar argument to that employed by Gernert \cite{GERNERT1989151} in proving that $\gamma(G)\chi(G) \leq \frac{1}{4}n^2$. If we denote by $\Delta$ to the maximal degree of $G$, it is clear that $i(G)\leq n-\Delta$ and, if $G$ is neither an odd cycle nor the complete graph, then $\chi(G)\leq \Delta$, and multiplying both expressions, we obtain: $i(G)\chi(G)\leq \Delta (n.-\Delta)$. The maximum of the right hand side of the inequality is for $\Delta =n/2$, and then $i(G)\chi(G) \leq \frac{1}{4}n^2$.
\end{proof}

\begin{remark}
    Observe that the inequality $i(G)\chi(G) \leq \frac{1}{4}n^2$ given in Theorem \ref{th:Omegachigamma} is a slight improve of \cite[Theorem 2]{GERNERT1989151}.
\end{remark}

Observe that by a greedy scheme to obtain the equality $\Omega_{uc}(G)=i(G)\chi(G)$, each element in every minimal independent dominant set $D$ must have in its neighborhood a vertex of each of the $\chi$ colors.

A hairy graph of $G$ is the resulting graph when a leaf is added as a new neighbor of each vertex of G. If $l$ leaf vertices are added to each vertex, we say that is a $l$-hairy graph.

\begin{proposition}
\label{prop:hairy}
Let $G$ be a graph with chromatic number $k=\chi(G)$ and $H$ its $(k+1)$--hairy graph. Then, any $\Omega_{uc}$--domination coloring of $H$ uses at least $k+1$ colors.
\end{proposition}

\begin{proof}
Let $c$ be an $\Omega_{uc}$--domination coloring of $H$ and $D$ an associated $\Omega_{uc}$--dominating set. We assume (reductio ad absurdum) that $c$ uses only $k$ colors (from $0$ to $k-1$). There is at least one vertex $v$ in $G$ such that $c(v)=0$ (in other case, $G$ could be colored with less than $k$ colors), and therefore $v \notin D$. Consequently, all its $k+1$ leaf vertices must be in $D$, and their contribution to $\Omega_{uc}(H)$ is greater than or equal to $k+1$. 

Another coloring $c'$ is defined in such a way that assigns color $0$ to those leaf vertices and colors the vertex $v$ with the new color $k$. If these leaf vertices are removed from $D$ and the vertex $v$ is added in it, it is straightforward to see that $w_{uc}(H,c') < \Omega_{uc}(H)$, which is a contradiction.
\end{proof}

Two interesting weight minimization problems arise in this context of domination over colored graphs. On one hand, we pose the corresponding question to \textsc{up--color domination $c$--number} but relative to weight. The first problem is the following:

\noindent \textbf{\textsc{up--color domination $c$--weight}}  \\
{\sc instance:} A graph $G=(V,E)$, a coloring $c$ and a positive integer $k$. \\
{\sc question:} Is $\omega_{uc}(G,c) \leq k$?

And the proof of the NP--completeness of that problem is obtained by mimicking the proof of Theorem~\ref{th:npcd}.  Thus, we have:

\begin{theorem} 
\label{th:npcw}
\textsc{\emph{up--color domination $c$--weight}} is an NP--complete problem, even if the graph is bipartite and $c$ is a $3$--coloring, and even for optimal coloring. 
\end{theorem}

On the other hand, a more general problem is set out for up--color domination weight of graphs and we again establish its NP--completeness nature.

\noindent \textbf{\textsc{up--color domination weight}}  \\
{\sc instance:} A graph $G=(V,E)$ and a positive integer $k$. \\
{\sc question:} Is $\Omega_{uc}(G) \leq k$?

\begin{theorem}
\textsc{\emph{ up--color domination weight}} is an NP--complete problem.
\end{theorem}

\begin{proof}
As in Theorem~\ref{th:npcd}, it is straightforward to check that \textsc{up--color domination weight} is in NP. Next, we give a reduction from \textsc{balanced maximum E2--satisfiability} (or shortly \textsc{balanced max--E2--sat}) which is known to be NP--complete~\cite{2018-ppw-saabm2s}. 

Let $I_{mE2s}$ be an arbitrary instance of \textsc{balanced max--E2--sat}, which is made up of a set $U=\{ u_1,u_2,\ldots, u_r\}$ of $r$ Boolean variables, a collection $C=\{c_1, c_2, \ldots, c_s\}$ of $s$ clauses on $U$ such that each clause contains exactly two literals and each variable occurs affirmatively and negatively equally often, and a positive integer $0< l\leq |C|$. Then, the goal of \textsc{balanced max--E2--sat} is to find a truth assignment over $U$ that simultaneously satisfies at least $l$ clauses of $C$.

We define an instance $I_{uw}$ of \textsc{up--color domination weight} given by a graph $G$ of $12  r  s +6 s$ vertices and the positive integer $k = 6  r s + 3 s + 2  (s-l)$  as follows.

The graph $G$ consists of different gadgets that represent the variables and the clauses of $I_{mE2s}$. On the one hand, each variable $u_i$ gives rise to a hairy bipartite graph that is denoted by $HB_{u_i}$ (see Figure~\ref{fig:npOmega}(a)). As we mentioned above, it is made up of the complete bipartite graph $K_{3s, 3 s}$ along with $6 s$ hairs, that is, an incident edge for each of its $6 s$ vertices. Thus, $HB_{u_i}$ contributes with $12 s$ vertices to $G$.

\begin{figure}[ht] 
	\centering
\includegraphics[width=0.8\textwidth]{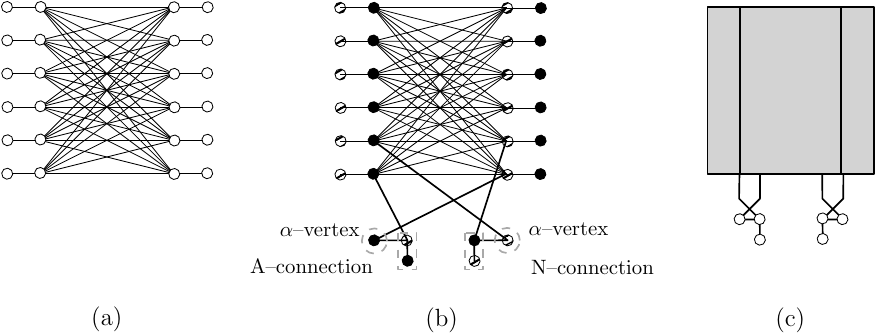}
	\caption{(a) The hairy bipartite graph $HB_{u_i}$ associated to $u_i$. (b) Two assemblies for $HB_{u_i}$. The affirmed and negated variable are represented by solid and striped vertices, respectively. The $\alpha$--vertex is enclosed by a dashed circle and the hair of the subgraphs of three vertices is marked with a dashed rectangle. (c) Schematic representation of (b).}
 \label{fig:npOmega}
\end{figure}

On the other hand, the clauses determine assemblies between the hairy bipartite graphs, and since those contain exactly two literals, these are assembled in pairs. Every literal is represented by a subgraph of three vertices as Figure~\ref{fig:npOmega}(b) shows, in which there are a distinguished vertex called $\alpha$--vertex and a hair or a pair of adjacent vertices (marked with a dashed circle and rectangle in Figure~\ref{fig:npOmega}(b), respectively). 
This subgraph is merged into the hairy bipartite graph following an A--connection or an N--connection depending on whether the clause contains the affirmed or negated variable, respectively. These mergers between subgraphs of three vertices and hairy bipartite graphs need as vertices of $HB_{u_i}$ as twice the number of clauses that contain $u_i$. Finally, the pair of hairy bipartite graphs is assembled by the adjacency of the two associated $\alpha$--vertices. Note that for every hairy bipartite graph, there exist as many A--connections as B--connections since every variable appears affirmatively and negatively equally often. Therefore, the clauses contribute with $6 s$ vertices to $G$, since the total number of subgraphs of $3$ vertices is $2  s$. 

Then, every Boolean variable of $U$ is represented in $G$ by a bipartite subgraph of the total graph, that is made up of a hairy bipartite graph and an even number of subgraphs of three vertices. One of its two classes of vertices as bipartite subgraph is associated with the affirmed variable (shown in Figure~\ref{fig:npOmega}(b) with solid vertices), and the other corresponds to the negated variable (depicted in Figure~\ref{fig:npOmega}(b) by striped vertices). 

Figure~\ref{fig:npOmega_ex} depicts the graph $G$ for a specific instance $I_{mE2s}$ with $r=s=4$. The clause collection of $I_{mE2s}$ is $C=\{\bar{u_1}u_3, u_1\bar{u_2}, u_2u_4, \bar{u_3}\bar{u_4}\}$ and its clauses are represented in $G$ in blue, red, orange, and green colors, respectively.

\begin{figure}[ht] 
	\centering
\includegraphics[width=0.7\textwidth]{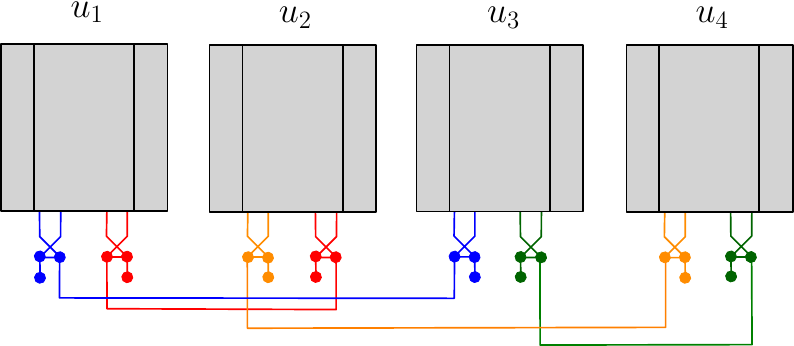}
	\caption{The graph $G$ obtained from an instance of  \textsc{balanced max--E2--sat} made up of four variables and the clause collection  $C=\{\bar{u_1}u_3, u_1\bar{u_2}, u_2u_4, \bar{u_3}\bar{u_4}\}$, whose clauses are represented in $G$ in blue, red, orange, and green colors, respectively.}
 \label{fig:npOmega_ex}
\end{figure}

It is clear that $I_{uw}$ is constructed in polynomial time in the size of $I_{mE2s}$, that is, $r$, $s$ and $l$.

Next, we prove that $I_{mE2s}$ admits a truth assignment $\Phi$ over $U$ that simultaneously satisfies at least the $l$ clauses of $C$ if and only if there is a coloring $c$ of $G$ such that it is possible to determine an up--color dominating $c$--set for the pair $(G,c)$ whose weight is smaller than or equal to $k$. Thus, we will conclude that \textsc{up--color domination weight} is an NP--complete problem. This equivalence of affirmative answers of both problems in general is supported by the link between true (resp. false) and color 1 (resp. color 0).

On the one hand, given such a truth assignment $\Phi$, we define a $3$--coloring $c$, starting with an auxiliary color assignment $c'$, as follows. If $\Phi(u_i)$=true, then $c'$ colors the vertices associated with the affirmed variable (resp. negated variable) with color 1 (resp. color 0), and vice versa if $\Phi(u_i)$ = false.  As a consequence, the $\alpha$--vertices are colored according to the value of $\Phi$ on the literals that symbolize, following the above rule (true and color 1, and false and color 0).  

\begin{figure}[ht] 
	\centering
\includegraphics[width=0.6\textwidth]{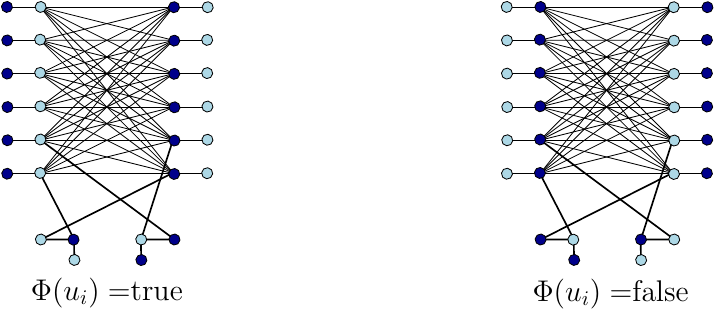}
	\caption{Color 1 or true literals (resp. color 0 or false literals) are represented by vertices in light blue (resp. dark blue).}
 \label{fig:npOmega_01}
\end{figure}

We highlight the fact that, up to now, $c'$ has assigned as many zeros as ones due to the balancing property of $I_{mE2s}$ (see Figure~\ref{fig:npOmega_01}). However, $c'$ is not a coloring yet, since the possible existence of clauses containing two literals with the same truth assignment leads to adjacent $\alpha$--vertices colored with the same color. In this case, the color of one of these $\alpha$--vertices changes to 2, and in this way we obtain the coloring $c$ of $G$. 

For the pair $(G,c)$, it is possible to obtain an up--color dominating $c$--set $D$ of weight at most $k = 6 r s +3 s + 2  (s-l)$ following the steps below. In a first stage, $D$ contains all vertices originally colored with 1 in $c'$, which entails a weight of $6  r  s +3 s$. In a second step, we focus our attention on the pairs of adjacent $\alpha$--vertices colored with the same color 1 in $c'$, which belong to $D$. As we have mentioned previously, one vertex of every such pair changes its color to 2 and so remains in $D$, and the other is removed from $D$. Thus, the weight of $D$ is not modified. Finally, the vertices colored 2 and coming from 0 in $c'$ are added to $D$. Note that there are at most $(s-l)$ such vertices since $\Phi$ simultaneously satisfies at least $l$ clauses. It is straightforward to verify that the vertex set $D$ is an up--color dominating $c$--set of weight at most $k = 2  (s-l)+ 6  r  s +3 s$.

On the other hand, let $c$ be a coloring of $G$ that provides an affirmative answer for $I_{uw}$, which means there exists an up--color dominating $c$--set $D$ of weight $w(D)$ at most $k$. Without loss of generality, we may assume that $c$ uses all colors from $0$ to a certain maximum color. Furthermore, it is obviously that each dominant vertex contributes with at least one unit to $w(D)$ and that, for each hair of $G$, at least one of its vertices must belong to $D$. By following the below steps, we define a truth assignment $\Phi$ over $U$ that simultaneously satisfies at least $l$ clauses of $C$. By language abuse, we indiscriminately refer to subgraphs of $G$ or to the components of $I_{mE2s}$ that represent:

\renewcommand{\theenumi}{\roman{enumi}}

\begin{enumerate}
\item First, we prove that there is at least one vertex in each $K_{3s, 3s}$ such that its color by $c$ is $0$.
\item Then, we assign a labeling to the vertices of $G$ following the property in the previous step. An important property of this labeling is that the number of vertices labeled with $1$ is exactly half of the vertices of $G$, that is, $6rs+3s$.  
\item Next, we compare the weight of $D$ with the sum of all labels in $V(G)$.
\item Finally, $\Phi$ is defined based on this labeling, and, by the bounds obtained in the previous step, we conclude that at least $l$ clauses are satisfied.
\end{enumerate}

\noindent {\sl Step i.} Since one of the two vertices of each hair of $G$ must belongs to $D$, it is clear that $w(D)$ is at least $6rs+2s$, that is, the number of hairs of $G$. Now, assuming on the contrary, let $HB_{u_i}$ be a bipartite subgraph such that no vertex of degree greater than $1$ (that is, vertex in $K_{3s, 3s}$) has color $0$ in $c$. Thus, all vertices of one of the classes of $K_{3s, 3s}$ as bipartite graph have color at least $1$, and then the color of the vertices in the other class must be at least $2$. Therefore, the domination or belonging to $D$ of these $3s$ vertices contributes with at least $3s$ units more to $w(D)$. In this way, the total weight of $D$ would be at least $6rs+2s+3s$, what implies the contradiction of $l=0$. So, there is at least one vertex in each $K_{3s, 3s}$ such that its color by $c$ is $0$.

\noindent {\sl Step ii.} Obviously, if one of the classes of $K_{3s, 3s}$ contains a vertex of color $0$, then all vertices in the other class have colors greater than $0$. Therefore, the class of each $K_{3s, 3s}$ containing some vertex of color $0$ is well defined. This vertex class of $K_{3s, 3s}$ is included in one of the classes of the bipartite graph that represents each Boolean variable (see Figure~\ref{fig:npOmega}(b)). Thus, the vertex labeling $\varphi$ on $G$ is defined by assigning the label $0$ to all vertices of this class (both $HB_{u_i}$ and subgraphs of three vertices) and the label $1$ to the other class (again, the whole class, both $HB_{u_i}$ and subgraphs of three vertices). 

Then, given $A \subseteq V(G)$, we define $\varphi(A) = \sum_{v \in A}\varphi(v)$. Note that half of the vertices of $G$ are labeled with $1$ and the other half with $0$ since $I_{mE2s}$ is a balanced instance. Hence, $\varphi(V(G))=6rs+3s$.

\noindent {\sl Step iii.} Through this step, $w(D)$ is compared with $\varphi(V(G))$ in the following way. First, we divide $V(G)$ into disjoint subsets and prove that in each one of these subsets $A$ it is verified that $\varphi(A) \leq w(A \cap D)$. Thus, for each clause $c_i$, with $1\leq i \leq s$, we define $A(c_i)$ as the vertex set of cardinal $14$ made up of the $6$ vertices in the gadget corresponding to $c_i$ (two subgraphs of three vertices) plus $8$ vertices in the two $HB_{u_i}$ to which $c_i$ is merged: the $4$ adjacent vertices to $c_i$ plus their adjacent vertices of degree 1. Figure~\ref{fig:npaci} shows vertex sets $A(c_i)$ for a clause including an affirmed variable (A--connection) and a negated variable (N--connection) (the remaining situations are similar), and the three possible cases according to the labeling given by $\varphi$ :

\renewcommand{\labelenumi}{(\alph{enumi})}
\begin{enumerate}
    \item $\varphi(\alpha_1) \not= \varphi(\alpha_2)$ and then $\varphi(A(c_i))=7$.  
    \item $\varphi(\alpha_1) = \varphi(\alpha_2)= 0$ and then $\varphi(A(c_i))=6$.
    \item $\varphi(\alpha_1) = \varphi(\alpha_2)= 1$ and then $\varphi(A(c_i))=8$.   
\end{enumerate}

\begin{figure}[ht] 
	\centering
\includegraphics[width=1\textwidth]{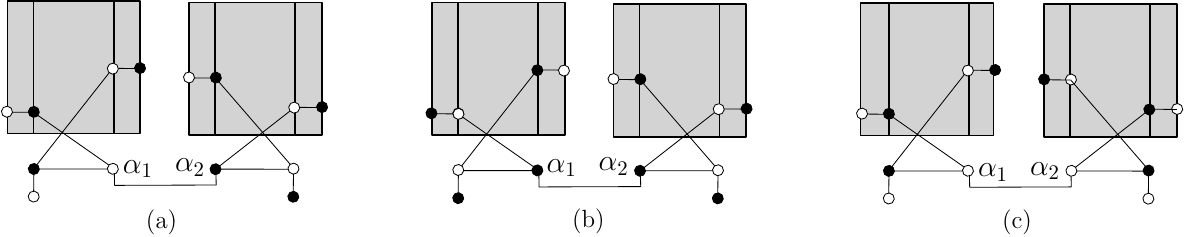}
	\caption{The value of $\varphi$ over white (resp. black) vertices is $1$ (resp. $0$). (a) $\varphi(\alpha_1) \not= \varphi(\alpha_2)$. (b) $\varphi(\alpha_1) = \varphi(\alpha_2)= 0$.  (c)  $\varphi(\alpha_1) = \varphi(\alpha_2)= 1$.} 
 \label{fig:npaci}
\end{figure}

Next, straightforward arguments are established to prove that $\varphi(A(c_i))\leq w(A(c_i) \cap D)$ for $1\leq i \leq s$. 
Since the induced subgraphs of $G$ by the vertex sets $A(c_i)$ have $6$ hairs, it holds that $w(A(c_i) \cap D)$ is greater or equal to $6$. Furthermore, at least one of the adjacent $\alpha$--vertices has to be colored with color greater than or equal to $1$, and therefore it is either dominant or the vertex that dominates it is colored with color at least $2$. This fact implies a minimum of one unit more of $w(A(c_i) \cap D)$. Thus, $\varphi(A(c_i))\leq w(A(c_i) \cap D)$ in the Cases (a) and (b).  

Since $\varphi(\alpha_1)=\varphi(\alpha_2)=1$, the point in the last Case (c) is that all the vertices of $K_{3,3}$ belonging to the same class than $\alpha_i$, for $i=1$, $2$, are colored with a color greater or equal to $1$. Thus, if one of the $\alpha$--vertices has color $0$, then, either its adjacent vertex in the corresponding subgraph of three vertices or the adjacent vertex to this one in $K_{3,3}$ is colored with at least color $2$. If no $\alpha$--vertex is colored with $0$, then one of them has a color greater or equal to $2$. In any case, $\varphi(A(c_i)) = 8 \leq w(A(c_i) \cap D)$.

In order to complete the division of $V(G)$ into disjoint subsets, the bipartite graph $G-\cup_{i=1}^s A(c_i)$ is considered. On the one hand, we point out that half of the vertices of this graph are labeled by $\varphi$ with 1 and the other half with $0$. On the other hand, we recall that each hair contributes with at least one unit to $w((G-\cup_{i=1}^s A(c_i)) \cap D)$ by means of one of its adjacent vertices. Thus, it is clear that $\varphi(G-\cup_{i=1}^s A(c_i)) \leq w((G-\cup_{i=1}^s A(c_i)) \cap D)$. Therefore, $\varphi(V(G)) = 6rs+3s \leq w(D)$.

Now, we bring to light one more property of the vertex sets $A(c_i)$ in Case (b) ($\varphi(\alpha_1)=\varphi(\alpha_2)=0$ and then $\varphi(A(c_i))=6$). As we have established previously, the $6$ hairs in the induced subgraph of $G$ by $A(c_i)$ contribute with at least $6$ units to $w(A(c_i) \cap D)$. Moreover, at least one of the adjacent $\alpha$--vertices has to be colored with a color greater than or equal to $1$, and therefore the domination or belonging to $D$ of its two adjacent vertices in the corresponding hairy bipartite graph demands $2$ units as minimum, since their colors cannot be zero (their labels by $\varphi$ are $1$). Then, it is verified that $\varphi(A(c_i)) +2 \leq w(A(c_i) \cap D)$ for each clause $c_i$ such that its $\alpha$--vertices are labeled with $0$ by  $\varphi$. 

Due to all of the above, if $n_0$ denotes the number of such clauses, we have on one hand that $\varphi(V(G)) +2n_0 = 6rs +3s +2n_0 \leq w(D)$, and on the other hand $w(D) \leq k= 6  r s + 3 s + 2  (s-l)$. Thus, $n_0 \leq s-l$.

\noindent {\sl Step iv.} Lastly, we are ready to define a truth assignment $\Phi$ for $I_{mE2s}$ in the following way: those literals represented by $\alpha$--vertices with label $1$ (resp. $0$) are said to be true (resp. false). It is straightforward to see that the number of non--satisfied clauses is exactly $n_0\leq s-l$, therefore, $\Phi$ simultaneously satisfies at least $l$ clauses of $C$.

Thus, the proof of the the NP--completeness of \textsc{up--color domination weight} problem is finished.
\end{proof}


Similarly to the result of Theorem~\ref{th:tree_lineal1} for the up--color domination $c$--number, in spite of having an exponential number of $\omega_{uc}$--dominating $c$--sets  (see Figure~\ref{fig:domsetexptree2}), the \textsc{up--color domination $c$--weight} problem becomes solvable in polynomial time on trees, as Theorem~\ref{th:weighted_tree_polinomial} states. This result is based on the following two facts: On the one hand, the optimal domination of the tree also determines the optimal domination for certain subtrees. On the other hand, we can exploit its acyclic structure to design an efficient polynomial--time algorithm. We will define a specific data structure and decision rules on how to optimally assign a dominant or submissive role to each vertex. In order to do this and by working in the directed context provided by the graph and its coloring (see the introduction of the paper), a depth--first search (DFS) from the source vertices will allow us to establish an optimal order for assigning roles to the vertices. 

\begin{figure}[ht] 
	\centering
\includegraphics[width=0.7\textwidth]{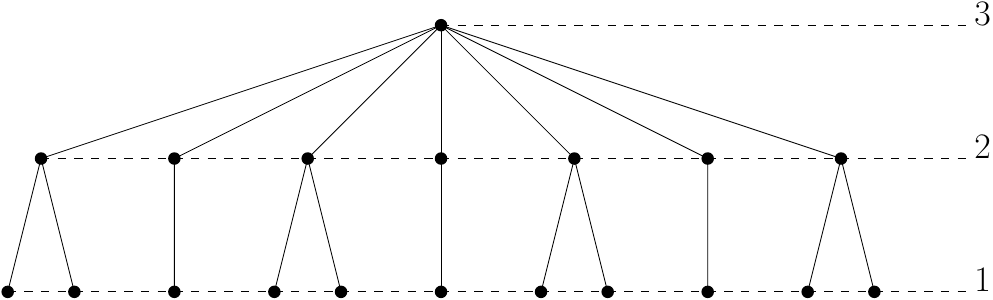}
	\caption{The number of $\omega_{uc}$--dominating $c$--sets is exponential as this example shows. The root of this tree, with color $3$, must be in any $\omega_{uc}$--dominating $c$--set, but then, it is possible to include either each vertex with the color $2$ and degree $2$ or its child. }
 \label{fig:domsetexptree2}
\end{figure}

We consider a tree $T$ and a coloring $c$ of $T$. First, given an $\omega_{uc}$--dominating $c$--set $D$, an edge $uv$ is chosen in $T$ such that $c(u) > c(v)$, and further assume that $u \in D$ while $v \notin D$. Then, whether $T_1, T_2, \ldots, T_k$ are the subtrees obtained from $T$ by removing $v$ and all the incoming edges to $u$, it is straightforward to see that $D \cap T_i$ is an $\omega_{uc}$--dominating $c$--set of the pair $(T_i, c|_{T_i})$ for $i=1, \ldots, k$.

Next, a recursive algorithm is developed to
solve the \textsc{up--color domination $c$--weight} problem for trees. Let $dT$ be the directed tree obtained by orienting the edges of $T$ starting from its vertex with the highest color. We denote by $F_c(dT)$, $S_c(dT)$, $C_c(dT)$ the set of source, sink, and confluence vertices of $dT$, respectively. Note that $S_c(dT) \cap C_c(dT)$ can be non--empty. Moreover, for each vertex $v$ of $dT$, we defined $dT_v$ as the directed subtree of $dT$ rooted in $v$. Since there are no cycles in $T$, it follows that $dT_v-\{v\}$ is a forest of disconnected subtrees.

This recursive procedure will visit each vertex $v$ to record and store values for dominating sets with minimum weight of $dT_v$ under the two possible assumptions: first, the vertex $v$ is dominant in the solution for $(T,c)$, and second, the said vertex is submissive in the solution for $(T,c)$. The procedure will start at each source vertex and, following the edges of $dT$, will visit the vertices until it ends at the sink vertices, whose associated values are known. In order to do this, we will define a data structure that stores the partial results obtained for each visited vertex. More specifically, for each vertex $v$ in the directed tree $dT$, the data structure $DS[v]$ is defined as $DS[v]= [[D_0(v), w_0(v)], [D_1(v),w_1(v)], op(v)]$, where $w_0(v)$ and $D_0(v)$ represent the minimum weight and a dominating set of $dT_v$, respectively, assuming that $v$ is dominant for $(T,c)$; and $w_1(v)$ and $D_1(v)$ represent the same elements, assuming that $v$ is submissive for $(T,c)$. The value $op(v)$ indicates the assumption with the lowest weight. Note that if $v$ is submissive for $(T,c)$, then each vertex $u \in N^+(v)\setminus C_c(dT)$ must necessarily be dominant, while the vertices $u \in N^+(v) \cap C_c(dT)$ can also be submissive (dominated by $N^-(u) \setminus \{v\}$). The final assignment of roles to each vertex of $T$ will require further verification and choice in the case of the vertices of $C_c(dT)$, which will be described later. With these ideas, we have the following recurrence relations:

\begin{enumerate}
	\item  $w_0(v)=c(v)+\sum\limits_{ u\in N^+(v)}^{} w_{op(v)}(u)$.
	\item   $D_0(v)=\{v\}\, \cup \bigcup\limits_{ u\in N^+(v)}^{} D_{op(u)} (u)$.
	\item  $w_1(v)=\sum\limits_{ u\in N^+(v)\cap  C_c(dT)} w_{op(v)}(u) + \sum\limits_{ u\in N^+(v) \setminus  C_c(dT)} w_0(u)$.
	\item $ D_1(v)=\bigcup\limits_{ u\in N^+(v)\cap  C_c(dT)} D_{op(v)}(u) \, \, \cup \bigcup\limits_{ u\in N^+(v) \setminus  C_c(dT)} D_0(u)$.
	\item $ op(v) = 0 \textup{\ if\ } w_0(v) \leq w_1(v) ; op(v) =1 \textup{\ otherwise} $.
\end{enumerate}

Note that if $v\in S_c(dT)$, then $ DS[v]=[ [\{v\},c(v) ], [\emptyset, 0], 1]$. On the other hand, if $v\in F_c(dT)$ then necessarily $v$ must be a dominant vertex, so we can avoid the computation of the second pair within $DS[v]$ and set $op(v)=0$.

The recurrence call is applied at most $n-1$ times (at least one sink vertex exists), and the memory read of $DS[v]$ is performed the same number of times. However, the memory cost is of the order $\mathcal{O}(n^2)$ to record all values $DS[v]$. In addition, numerous set unions need to be calculated even though they are disjoint sets (so with constant computational cost), as well as set intersections and set differences, so the computational cost of the calculation is of the order $\mathcal{O}(n\ log(n))$, and the memory cost is $\mathcal{O}(n^2)$.

Note that if $op(v)=1$ in $DS[v]$ (the optimum is obtained for $v$ as a submissive vertex), it is necessary for there to exist $u\in N^-(v)$ with $op(u)=0$ (the vertex $v$ is dominated by a vertex $u$) to be able to take the said optimum. For vertices $v\notin C_c(dT)$, there is only one option for domination (or none if it is a source vertex). In contrast, vertices $v\in C_c(dT)$ require a comparative analysis of which vertex $u\in N^-(v)$ would be the best candidate to dominate $v$, in the sense that the resulting total weight increase is minimal compared to taking vertex $v$ as dominant. Let us see how to do this comparative analysis.

Let $v\in C_c(dT)$ with $op(v)=1$ and $op(u)=1$ for all $u\in N^-(v)$. One possibility is to make $v$ dominant, which increases the total weight by $w_0(v)-w_1(v)$. Another possibility is to take one of the vertices $u$ incident to $v$ as dominant, with a consequent weight increase of $w_0(u)-w_1(u)$. The optimal choice is to take $$argmin(\{w_0(v)-w_1(v), w_0(u_1)-w_1(u_1),\dots, w_0(u_r)-w_1(u_r)\})$$ for all $u_i\in N^-(v)$, with $r=|N^-(v)|$.

Moreover, it is important to note that given $v, u \in C_c(dT)$, it can happen that $dT_v$ is a subgraph of $dT_u$, so the decision on the optimal choice in $v$ must be made after the decision in $u$. In such a case, we say that the optimal choice in $u$ precedes the optimal choice in $v$. 
To establish the precedence relationships among the confluence vertices, we proceed as follows: a depth--first search (DFS) is performed starting from each source vertex of the directed tree. A global counter records the visit order, which is not reset when a new DFS is initiated from another source vertex. Thus, the precedence relationship between the vertices in $C_c(dT)$ is based on the moment the confluence vertices are last visited. This approach ensures that the resulting order accurately captures the hierarchical structure and dependencies within the tree.

After all these considerations, an optimal solution to the weighted domination problem for $(T, c)$ is obtained for the dominating set formed by $D=\bigcup\limits_{v\in F_c(dT)} D_0(v)$, which must be modified and updated for each $v\in C_c(dT)$ that is neither dominant nor dominated, following the order given by the previous precedence relation. Thus, for example, if the vertex $v$ changes from submissive to dominant, then we update the set $D$ in two steps: $D=D\setminus D_1(v)$ and $ D=D\cup D_0(v)$. In this way, the computational cost to calculate $D$ is of the order $\mathcal{O}(|C_c(dT)|\, n\ log(n))= \mathcal{O}(n^2\ log(n))$.

Consequently, we have obtained the following theorem:

\begin{theorem} \label{th:weighted_tree_polinomial}
Given a pair $(T,c) $, where $T$ is a tree and $c$ is a coloring of $T$, it is possible to obtain an {\em  $\omega_{uc}$--dominating $c$--set of $T$}  with computational cost $\mathcal{O}(n^2\ log(n))$ and memory cost $\mathcal{O}(n^2)$.
	
\end{theorem}

\begin{figure}[ht]
    \centering
    \includegraphics[width=0.7\textwidth]{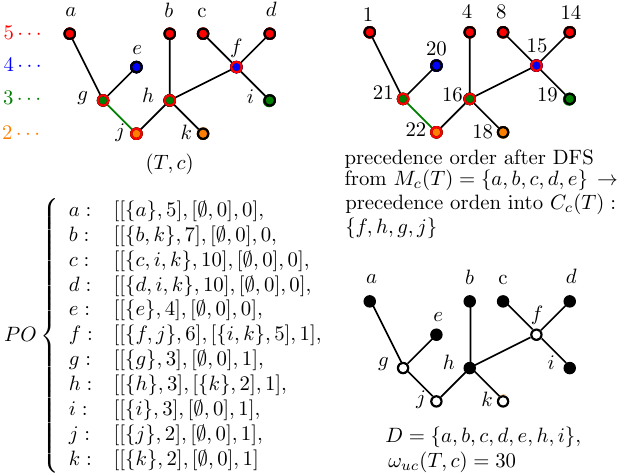}
    \caption{Illustrative example showing (top left) a tree graph and its coloring, (bottom left) the memory record values \( DS \), (top right) the precedence order resulting from the DFS algorithm, and (bottom right) the optimal dominating set, as described in Theorem \ref{th:weighted_tree_polinomial}}
    \label{fig:illustrative_example}
\end{figure}

Figure \ref{fig:illustrative_example} provides an illustrative example of the algorithm described in Theorem \ref{th:weighted_tree_polinomial}. It presents a tree graph \( (T, c) \) and its coloring in the top left section. In the bottom left section, the values of the memory record \( DS \) calculated for each vertex are shown, reflecting the domination decisions and their respective weights. The top right section details the precedence order resulting from applying the DFS algorithm starting from the source vertices, while the bottom right section represents the optimal dominating set obtained. This example allows us to visualize the process described in the theorem, illustrating how decisions are structured and optimized based on the calculated values and the determined hierarchical order.

	$c'(v)= \left\{ \begin{array}{lc}
             0 & \textrm{if }  v\in S_w, \\
             1 & \textrm{if } v\in S_b \textrm{ or } v\in D_{(b,i)}\ \  \forall i=1,\dots,k, \\
             2 & \textrm{if }v\in D_{(w,i)}
                          \end{array}
   \right.$

\section{Conclusions and open problems} 
\label{sec:cop}

The concept of up--color domination has been introduced in a natural way from the notions of coloration and domination in graphs.

The exploration of domination in graphs, as emphasized by Haynes et al.~\cite{2023-haynes-domination}, continues to inspire new directions and challenges in this field.

Following that purpose, in this work, we have introduced a natural concept of domination in the context of colored graphs (where each color assigns a weight to the vertices of its class), termed {\em up--color domination}, where a vertex dominating another must be heavier than the other. That idea allows to define two new parameters measuring the size or the weight of minimal dominating sets. For both concepts, finding an optimal set is an NP--hard problem for general graphs but polynomial for trees.

Addionally, some new definitions could be given and will be the subject of some future works. For instance, 
the minimum $\gamma_{uc}(G,c)$ among all possible colorings $c$ is called {\em the up--color domination number of $G$} and denoted by $\Gamma_{uc}(G)$. Every coloring $c$ verifying $\Gamma_{uc}(G)=\gamma_{uc}(G,c)$ is named {\em $\Gamma_{uc}$--domination coloring} and the associated $\gamma_{uc}$--dominating $c$--set is called  {\em $\Gamma_{uc}$--dominating set}.

The following problems are proposed as avenues for future research:
\begin{problem}
Given a graph $G$ and a natural number $k$, determine the value $$\gamma_k'(G)=\min \{\gamma_{uc}(G,c)\, /\,  c \mbox{ is a $k$--coloring of } G \}$$ and indentify the coloring(s) $c'$ such that $\gamma_{uc}(G,c')=\gamma'_k(G)$.

Analogously, compute the weight $$\omega'(G)=\min \{\omega_{uc}(G,c)\, /\,  c \mbox{ is an optimal coloring of } G \}$$ and find the coloring(s) $c'$ such that $\omega_{uc}(G,c')=\omega'(G)$.

\end{problem}

To the light of the well--known inequality $\gamma (G) \leq i(G) \leq \alpha (G)$, given a natural number $k$ such that $k \geq \chi(G)+1$, it is easy to check the new inequality with some of the new parameters $\gamma(G) \leq \gamma_k'(G) \leq i(G) \leq \chi(\overline{G})$.

\begin{problem}
Investigate the relationship between the previously defined parameters and $\gamma_{uc}(G,c)$.
\end{problem}

\begin{problem}
Develop algorithms to compute $\gamma_{uc}(G)$ and $\chi_{uc}(G)$ for specific graph families and analyze the computational complexity of these problems for general graphs.
\end{problem}

\begin{problem}
Actually, in the proof of Theorem~\ref{th:Omegachigamma} we never use more than $2\chi(G)-1$ colors. Is it true that there exists always an $\Omega_{uc}$--domination coloring $c$ of $G$ ($\omega_{uc}(G,c)=\Omega_{uc}(G)$) using fewer than $2\chi(G)$ colors?
\end{problem}

\begin{problem}
Find better bounds for the number of colors used by an  $\Omega_{uc}$--domination coloring:

\begin{itemize}
    \item General graphs.
    \item Bipartite graphs.
    \item Trees.
\end{itemize}
\end{problem}

\begin{problem}
Is it true for any tree $T$ that $\Omega_{uc}(T)\leq \gamma_r(T)$?
\end{problem}

Additionally, identify pairs of integers  $a, b \geq 2$  such that there exists a graph $G$  with $\Omega_{uc}(G)=a$ y $\gamma_r(G)=b$.

\begin{problem}
Is determining $\Omega_{uc}(G)$ NP--complete for the following classes of graphs?
\begin{itemize}
    \item Trees.
    \item Bipartite graphs.
    \item Interval graphs.
\end{itemize}
\end{problem}

These problems highlight key challenges in the study of these parameters and their practical computation in graph theory, offering a roadmap for further research in this area.

\bibliographystyle{plain}

\bibliography{refs.bib}

\printnomenclature

\end{document}